\documentclass[preprint,12pt,3p]{elsarticle}

\usepackage{amssymb}
\usepackage{amsthm}
\usepackage{amsmath}
\usepackage{enumerate}
\newcommand{\Mod}[1]{\ (\mathrm{mod}\ #1)}
\theoremstyle{plain}
\newtheorem{theorem}{Theorem}
\newtheorem{lemma}[theorem]{Lemma}

\newtheorem{corollary}[theorem]{Corollary}

\newtheorem{definition}[theorem]{Definition}
\newtheorem{proposition}[theorem]{Proposition}
\newtheorem{example}[theorem]{Example}
\newtheorem*{proof1}{Proof of Theorem 13}

\begin{document}

\begin{frontmatter}

\title{Dold's Theorem from Viewpoint of\\ Strong Compatibility Graphs}

\author[label1, label2]{Hamid Reza Daneshpajouh}
\address[label1]{School of Mathematics, Institute for Research in Fundamental Sciences (IPM), P.O. Box 19395-5746, Tehran, Iran.}
\address[label2]{Moscow Institute of Physics and Technology, Institutsky lane 9, Dolgoprudny, Moscow region, 141700}



\ead{hr.daneshpajouh@phystech.edu}


\begin{abstract}
Let $G$ be a non-trivial finite group. The well-known Dold's theorem states that:
There is no continuous $G$-equivariant map from an $n$-connected simplicial $G$-complex to a free simplicial $G$-complex of dimension at most $n$. In this paper, we give a new generalization of Dold's theorem, by replacing ``dimension at most $n$" with a sharper combinatorial parameter. Indeed, this parameter is the chromatic number of a new family of graphs, called strong compatibility graphs, associated to the target space. Moreover, in a series of examples, we will see that one can hope to infer much more information from this generalization than ordinary Dold's theorem. In particular, we show that this new parameter is significantly better than the dimension of target space ``for almost all free $\mathbb{Z}_2$-simplicial complex." In addition, some other applications of strong compatibility graphs will be presented as well. In particular, a new way for constructing triangle-free graphs with high chromatic numbers from an n-sphere $\mathbb{S}^n$, and some new results on the limitations of topological methods for determining the chromatic number of graphs will be given.
\end{abstract}


\end{frontmatter}

\section{Introduction}
From the perspective of transformation groups, the famous Borsuk-Ulam theorem states that:
There is no continuous $\mathbb{Z}_2$-equivariant map from the $n$-sphere $\mathbb{S}^{n}$ to the $m$-sphere $\mathbb{S}^{m}$, whenever $n> m$. It is known that the Borsuk-Ulam theorem has a lot of generalizations; see the excellent survey of Steinlein~\cite{Steinlein}. Probably the best generalization of the classical Borsuk-Ulam theorem is Dold's theorem, as it has found various applications in many directions, such as chromatic numbers of hypergraphs~\cite{Daneshpajouh2, mat1, Sar}, topological Tverberg-type results~\cite{Bla}, and fair division problems~\cite{Vu}. Here is the precise statement of Dold's theorem.
\begin{theorem}[Dold's theorem~\cite{Dold}]
Let $G$ be a non-trivial finite group, $\mathcal{K}$ an $n$-connected simplicial $G$-complex, and $\mathcal{L}$ a free simplicial $G$-complex of
dimension at most $n$. Then there is no continuous $G$-equivariant from $||\mathcal{K}||$ to $||\mathcal{L}||$.
\end{theorem}
The Borsuk-Ulam theorem has also a handy combinatorial consequence, called octahedral Tucker's lemma~\cite{mat}. Several extensions of the octahedral Tucker lemma with fascinating applications in various area such as graph colorings, and fair division problems are known. We refer the interested reader to~\cite{De, Fan, Ziegler} for generalizations, and~\cite{Ali, Daneshpajouh3, De, mat, pal} for applications. Not long ago, a new generalization of octahedral Tucker's lemma, called $G$-Tucker's lemma, was introduced in~\cite{Daneshpajouh}. Moreover, as an application of that generalization, a new method for constructing graphs with high chromatic number and small clique number was given. To recall the lemma, we need to make some definitions and conventions. 

From now on, $G$ stands for a non-trivial finite group with $0\notin G$. Furthermore, its identity element will be denoted by $e$, and we define $g\cdot 0= 0$ for all $g\in G$.  Consider the $G$-poset(for definition, see the preliminaries section) $G\times\{1,\ldots , n+1\}$ with natural $G$-action, $h\cdot (g, i)\to (hg, i)$, and the order defined by $(h, x)\prec_1 (g, y)$ if $x < y$ (in $\mathbb{N}$). Also, let ${\left(G\cup\{0\}\right)}^n\setminus\{(0,\ldots , 0)\}$ be the $G$-poset whose action is $g\cdot (x_1, \ldots, x_n) = (g\cdot x_1, \ldots, g\cdot x_n)$, and the order relation is given by:
$$x=(x_1, \ldots, x_n)\preceq y=(y_1, \ldots, y_n),$$
if for every $i\in\{1, \ldots, n\}$, $x_i\neq 0$ implies $x_i=y_i$. We are now in a position to recall $G$-Tucker's lemma.
\begin{lemma}[$G$-Tucker's lemma~\cite{Daneshpajouh}]
Suppose that $n$ is a positive integers, $G$ is a non-trivial finite group, and 
$$\lambda : {\left(G\cup\{0\}\right)
}^n\setminus\{(0,\ldots , 0)\}\to G\times\{1, \ldots, (n-1)\}$$
is a map such that $\lambda (g\cdot x)= g\cdot\lambda(x)$ for all $g\in G$ and all $x\in{\left(G\cup\{0\}\right)
}^n\setminus\{(0,\ldots , 0)\}$. Then there exist two elements $x, y\in{\left(G\cup\{0\}\right)
}^n\setminus\{(0,\ldots , 0)\}$, and $e\neq g\in G$  such that $x\prec y$ and $\lambda (x) = g\cdot\lambda (y)$.
\end{lemma}
It is worth mentioning that octahedral Tucker's lemma can be simply derived from $G$-Tucker's lemma by putting $G=\mathbb{Z}_2$.
\subsection{Aims and Objectives}  
The main purpose of this paper is to present a new generalization of Dold's theorem by replacing the dimension of the target space with a sharper combinatorial parameter. In order to achieve this goal, first a new family of graphs, called strong compatibility graphs, will be introduced. In fact, a strong compatibility graph $\widetilde{C}_P$ is a graph associated to a $G$-poset $P$; a partially ordered set equipped with a group action $G$. Then, we will establish a connection between a topological parameter associated to $P$, the chromatic number of its strong compatibility graph, and its dimension. Consequently, this will lead to our desired result. Later, in a series of examples we will see that how much this generalization can be stronger than ordinary Dold's theorem. In particular, it will be shown that this new version of Dold's theorem is significantly better than ordinary Dold's theorem for ``almost all free $\mathbb{Z}_2$-simplicial complex." A generalization of $G$-Tucker's lemma will be presented as well. In addition, some other applications of strong compatibility graphs, such as constructing triangle-free graphs with high chromatic numbers from the $n$-sphere $\mathbb{S}^n$, and some new results on the limitations of topological methods for determining the chromatic number of graphs will be given.

\section{Preliminaries}
In this section, we collect some definitions and auxiliary results needed later in the paper. We assume that the reader is familiar with standard definitions and concepts of simplicial complexes.
For background on simplicial complexes we refer the reader to~\cite{matousek2008using}. Throughout the paper, the term ``simplicial complex" always means abstract simplicial complex, the geometric realization of a simplicial complex $\mathcal{K}$ is denoted by $||\mathcal{K}||$, and the symbols $[n]$ stands for the set $\{1,\ldots, n\}$.
\vspace{.2cm}

\textbf{Graphs:} In this paper, all graphs are finite, simple and undirected. The chromatic number $\chi (H)$ of a graph $H$ is the smallest number of colors needed to color the vertices of $H$ so that no two adjacent vertices share the same color.
\vspace{.2cm}

\textbf{Posets and $G$-posets: }A pair of elements $a, b$ of a partially ordered set $(P, \preceq)$ (poset for short)  are called comparable if either $a\preceq b$ or $b\preceq a$. A subset of a poset in which each two elements are comparable is called a chain. To any poset $(P, \preceq )$ we associate its order complex $\Delta (P)$, whose simplices are given by chains in $P$. A $G$-poset is a poset together with a $G$-action on its elements that preserves the partial order, i.e, $x\preceq y \Rightarrow g\cdot x\preceq g\cdot y$. A $G$-poset $P$ is called free $G$-poset, if for all $x$ in $X$, $g\cdot x=x$ implies $g=e$. One can see that, if $P$ is a free $G$-poset then its order complex $\Delta (P)$ is a free simplicial $G$-complex.
\vspace{.2cm}

\textbf{Connectivity and $G$-index: }Let $k\geq 0$ be a non-negative integer. A topological space $X$ is $k$-connected if its homotopy groups $\pi_0(X), \pi_1(X), \ldots, \pi_{k}(X)$ are all trivial. Also, for convenience, we make conventions that $(-1)$-connected means nonempty, and the empty set is $(-\infty)$-connected. The largest $k$, if it exists, that $X$ is $k$-connected is called the connectivity of $X$, and denoted by $conn(X)$. A simplicial complex $\mathcal{K}$ is called $k$-connected if $||\mathcal{K}||$ is $k$-connected. Similarly, a poset $P$ is called $k$-connected if $||\Delta (P)||$ is $k$-connected.

For an integer $n\geq 0$ and a group $G$, an $\mathbb{E}_nG$ space is the geometric realization of an $(n-1)$-connected free $n$-dimensional simplicial $G$-complex. For a $G$-space $X$, we define
$$ind_{G} X = \min\{n |\,\text{there is a continuous}\,G\text{-equivariant map}\, X\to \mathbb{E}_nG\}.$$
It is worth pointing out that the value of $ind_{G} X$ is independent of which $\mathbb{E}_nG$ space is chosen, because any of them $G$-equivariantly maps into any other, see~\cite[section 6.2]{matousek2008using} for details. For a concrete example, one can see that the geometric realization of order complex of $(G\times [n+1], \preceq_1)$, $||\Delta (G\times [n+1])||$, is an example of $\mathbb{E}_nG$ space. It is worth noting that $\Delta(G\times [n+1])$ is the standard $(n+1)$-fold join $\underbrace{G*G*\dots *G}_{n+1}$. Let us finish this section by listing some basic properties of $ind_{G} X$. 
\begin{proposition}[\cite{matousek2008using}]
Let $G$ be a non-trivial finite group, and let $X, Y$ be $G$-spaces.
\begin{enumerate}
    \item If there is a continuous $G$-equivariant map from $X$ to $Y$, we have $ind_{G} X\leq ind_{G} Y$.
    \item For every $\mathbb{E}_nG$ space, $ind_{G} \mathbb{E}_nG = n$.
    \item If $X$ is $k$-connected, then $k+1\leq ind_{G} X$.
\end{enumerate}
\end{proposition}
\section{Strong compatibility graphs and Dold's Theorem}
To propose a new method for finding topological lower bounds for the chromatic numbers of graphs, compatibility graphs were introduced in~\cite{Daneshpajouh1}. Later, some other applications of this new family of graphs, such as a new proof of the well-known Kneser conjecture, and a new way of constructing graphs with high chromatic numbers and small clique numbers, have been found~\cite{Daneshpajouh}. Moreover, in order to attack to a generalization of the famous Hedetniemi conjecture, a new version of this concept has been introduced for hypergraphs. To see this and some other applications, see~\cite{Daneshpajouh2}. We should emphasize that some variant of compatibility graphs, for $G=\mathbb{Z}_2$, were defined before by several authors~\cite{csorba2007homotopy, walker1983graphs, vzivaljevic2005wi}. Let us begin by recalling compatibility graphs from~\cite{Daneshpajouh1}.
\begin{definition}[Compatibility graph]
Let $P$ be a $G$-poset. The compatibility graph of $P$, denoted by $C_P$, is the graph $C_P$ with vertex set $P$, and two elements $x, y\in P$ are adjacent if there is an element $g\in G\setminus\{e\}$ such that $x$ and $g\cdot y$ are comparable in $P$.
\end{definition}
 It is worth noting that the main idea of this paper was inspired from the following theorem.
\begin{theorem}[\cite{Daneshpajouh1}]
If $P$ is a finite free $G$-poset, then
$$ind_{G} ||\Delta (P)||+ |G|\leq\chi\left(C_P\right).$$
\end{theorem}
Actually, above theorem shows us a connection between the connectivity of a $G$-poset and the chromatic number of its compatibility graph, as for any $G$-space we always have $Ind_{G} X\geq conn(X)+1$. So it was natural to think that there might be a relation between Dold's theorem and above statement. At first we thought we could replace ``dimension of at most $n$", in the stament of Dold's theorem,  with the chromatic number of the compatibility graph of a suitable $G$-poset. But, there were two issues with that: One is that, the difference between $\chi\left(C_P\right)$ and $ind_{G}||P||$ becomes larger, as the size of $G$ increased. The other one is that, it can also be an arbitrary large gap between the dimension of a $G$-poset and the chromatic number of its compatibility graph. For example, the compatibility graph of $(G\times [n], \preceq_1)$, $C_{G\times [n]}$, is isomorphic with the complete graph $K_{n|G|}$. Thus
$\chi\left(C_{G\times [n]}\right)= n|G|$,
which can be arbitrary larger than $\dim (G\times [n])= n-1$. Therefore, for our purpose, a new version of the compatibility graph is needed. Let us define strong compatibility graph as follows.
\begin{definition}[Strong compatibility graph]
Let $P$ be a $G$-poset. The strong compatibility graph of $P$, denoted by $\widetilde{C}_P$, is the graph $\widetilde{C}_P$ with vertex set $P$, and two elements $x, y\in P$ are adjacent if there is an element $g\in G\setminus\{e\}$ such that $x$ and $g\cdot y$ are comparable in $P$ and $y\notin [x]$, where $[x]=\{g\cdot x : g\in G\}$.
\end{definition}
Despite of compatibility graphs, in the next lemma, we will see that the chromatic number of strong compatibility graph of a finite $G$-poset has a good connection to its dimension.
\begin{lemma}
If $P$ is a finite $G$-poset, then
$$\chi\left(\widetilde{C}_P\right)\leq\dim(P)+1.$$
\end{lemma}
\begin{proof}
Define 
\begin{align*}
 c: & \widetilde{C}_P\longrightarrow\{1, \ldots, \dim(P)+1\}\\
  & p \longmapsto \max\{\left| p \prec p_1 \prec \cdots \prec p_m\right| : p_i\in P\}.
\end{align*}
We claim that $c$ is a proper coloring of $\widetilde{C}_P$. If $p$ and $q$ are connected in $\widetilde{C}_P$, then there is a $e\neq g\in G$ such that $p$ and $g\cdot q$ are comparable in $P$. Note that $p\neq g\cdot q$, since by definition of $\widetilde{C}_P$, $p\notin\{h\cdot q : h\in G\}$. Without loss of generality assume that $p \prec g\cdot q$. Also, let $c(q) = t+1$ which means there is a chain of length $t$ in $P$ of the following form 
$$ q \prec q_1 \prec \cdots \prec q_t.$$
By multiplying the previous chain by $g$, we get
$$ g\cdot q \prec g\cdot q_1 \prec \cdots \prec g\cdot q_t.$$
On the other hand, we have $p \prec g\cdot q$. Now, by the transitivity of $\prec$
$$p \prec g\cdot q \prec g\cdot q_1 \prec \cdots \prec g\cdot q_t,$$
which means $c(p) > c(q)$. Therefore $c$ is a proper coloring of $\widetilde{C}_P$. Thus, $$\chi\left(\widetilde{C}_P\right)\leq\dim(P)+1.$$
\end{proof}
The following inequality is an analogue of the inequality given in Theorem $5$ with this advantage that the size of $G$ is replaced by $1$. 
\begin{theorem}
If $P$ is a finite free $G$-poset, then
$$ind_{G} ||\Delta (P)||+ 1\leq\chi\left(\widetilde{C}_P\right).$$
\end{theorem}
\begin{proof}
Let $c : \widetilde{C}_P\to [m]$ be a proper coloring of $\widetilde{C}_P$ with $m$ colors. In the following, we will show that this coloring induces a simplicial $G$-equivariant map
\begin{align*}
  \lambda: & \Delta (P)\longrightarrow \Delta (G\times [m])\\
  & x\longmapsto (\lambda_{1}(x), \lambda_{2}(x)).
\end{align*}
First, we divide $P$ into equivalence classes, the orbits under the $G$-action, where for every $x\in P$ each class $[x]$ contains all elements $g\cdot x$, where $g\in G$, i.e, $[x] = \{g\cdot x : g\in G\}$. We pick one element from each class $[x]$ as a representative, say $x^{\prime}$, and we set $\lambda (x^{\prime}) = (e, c(x^{\prime}))$. Then, we extend $\lambda$ on the remaining elements of each class in the only possible way that it preserves the $G$-action and $\lambda_{2}(x) = c(x^{\prime})$ for each $x\in[x^{\prime}]$. In other words, we define $\lambda(g\cdot x^{\prime}) = (g, c(x^{\prime}))$ for each $g\in G$. 
Let us to verify that $\lambda$ is a well-defined function. For this purpose, we need to show that any point in $[x^{\prime}]$, say $x$, can be uniquely represented as $x= g\cdot x^{\prime}$ for some $g\in G$. Now, suppose that $x= g\cdot x^{\prime}= h\cdot x^{\prime}$ for some $g, h\in G$. Then $(h^{-1}g)\cdot x^{\prime}= x^{\prime}$. So $h^{-1}g=e$, as $P$ is a free $G$-poset. Thus, $h=g$.

Next, we show that $\lambda$ is a simplicial $G$-equivariant map. Clearly, by the definition of $\lambda$, this map preserves the $G$-action, i.e, for each $g\in G$ and $x\in P$, $\lambda(g\cdot x)=g\cdot \lambda(x)$. So, to prove our claim we just need to show that $\lambda$ is a simplicial map, i.e, takes any simplex to a simplex. Note that $\sigma\subseteq G\times [m]$ is a simplex of $\Delta (G\times [m])$ if and only if it contains no two different elements with the same second entries. Therefore, $\lambda$ is a simplicial map if and only if for all comparable elements $x, y$ in $P$, if $\lambda_2 (x)=\lambda_2(y)$, then $\lambda_1(x)=\lambda_1(y)$.

Now, let $x$ and $y$ be distinct elements of $P$ with $x\prec y$. Moreover, assume that $\lambda (x) = (g, c(x^{\prime}))$ and $\lambda (y) = (h, c(y^{\prime}))$, where $x^{\prime}$ and $y^{\prime}$ are representatives of classes $[x]$ and $[y]$, respectively, and $c(x^{\prime})= c(y^{\prime})$.  To finish the proof, we need to show that $g=h$. Suppose, contrary to our claim, that is $g\neq h$. By definition of $\lambda$, $x^{\prime} = g^{-1}\cdot x$ and $y^{\prime} = h^{-1}\cdot y$. So, $$\underbrace{g^{-1}\cdot x}_{x^{\prime}}\prec g^{-1}\cdot y = (g^{-1}h)\underbrace{h^{-1}\cdot y}_{y^{\prime}}.$$ Therefore, taking into account that $g^{-1}h\neq e$, we will conclude that $x^{\prime}$ and $y^{\prime}$ are adjacent in $\widetilde{C}_P$, as soon as we prove that $y^{\prime}\notin [x^{\prime}]$. And, this contradicts the fact that $c$ is a proper coloring of $\widetilde{C}_P$. To obtain a contradiction, suppose that $y^{\prime}\in [x^{\prime}]$. This implies that $y\in [x]$ as well. So, there is a nontrivial element $s\in G$ such that $y= s\cdot x$. Thus $x \prec s\cdot x$, as $x\prec y$. By multiplying both sides of previous inequality by $e, s, \ldots , s^{|G|-1}$, respectively, we get
\begin{align*}
x &\prec s\cdot x\\
s\cdot x &\prec s^2\cdot x\\
&\vdots\\
s^{|G|-1}\cdot x &\prec s^{|G|}\cdot x=x.
\end{align*}
Now, by the transitivity of $\prec$, $x \prec x$, which is impossible. In summary, until yet, we have concluded that $\lambda$ is a $G$-simplicial map. This map naturally induces a continuous $G$-equivariant map from $||\Delta (P)||$ to $||\Delta\left (G\times [m]\right)||$. Therefore, according to Proposition $3$, $$ind_{G} ||\Delta (P)||\leq ind_{G} ||\Delta\left (G\times [m]\right)|| = m-1.$$ Consequently,
$ind_{G} ||\Delta (P)||+ 1\leq\chi (\widetilde{C}_P)$. This is the desired conclusion.
\end{proof}

As corollaries of Lemma 7 and Theorem 8, we will state and prove a generalization of Dold's theorem, and $G$-Tucker's lemma. Before proceeding, let us recall a definition. The face poset of a simplicial complex $\mathcal{K}$ is the poset $P(\mathcal{K})$, which is the set of all nonempty simplices of $\mathcal{K}$ ordered by inclusion. Moreover, if $\mathcal{K}$ is a free $G$-simplicial complex, then $G$ induces a free action on the poset $P(\mathcal{K})$, and consequently turns it to a free $G$-poset. Now, we are in a position to state and prove a common generalization of Dold's theorem and $G$-Tucker's lemma.

\begin{theorem}
If $G$ is a non-trivial finite group, $P$ an $n$-connected $G$-poset, and $\mathcal{Q}$ a free $G$-poset with $\chi\left(\widetilde{C}_P\right)\leq n+1$, then:
\begin{enumerate}[a)]
    \item There is no continuous $G$-equivariant map from $||\Delta (P)||$ to $||\Delta (Q)||$.
    \item  If $\lambda : P\to Q$ is a $G$-map, i.e, $\lambda(g\cdot x)= g\cdot\lambda(x)$, then there exist comparable elements $x, y$ in $P$ such that $\lambda(x)$, and $\lambda(y)$ are not comparable in $Q$.
\end{enumerate}
\end{theorem}
\begin{proof}
Suppose, contrary to our claim, that there is a continuous $G$-equivariant map\linebreak
$\Psi : ||\Delta (P)||\to||\Delta (Q)||$. Therefore, by Proposition $3$ and Theorem $8$
\begin{align*}
n+1  = conn (||\Delta (P)||) + 1 & \leq Ind_{G}||\Delta (P)||\\
& \leq Ind_{G}||\Delta (Q)||\leq\chi\left(\widetilde{C}_{P(\mathcal{L})}\right)-1\leq n,
\end{align*}
which is impossible. The part $(b)$ is easily deduced from part $(a)$. Indeed, if $\lambda(x)$ and $\lambda(y)$ were comparable for every comparable elements $x, y$ in $P$, then $\lambda$ would naturally induce a continuous $G$-equivariant map from $||\Delta (P)||$ to $||\Delta (Q)||$. 
\end{proof}
Lemma $7$ ensures us that Dold's theorem is a consequence of the part $(a)$. More precisely, let $G$ be a finite non-trivial group, $\mathcal{K}$ an $n$-connected simplicial $G$-complex, and $\mathcal{L}$ a free simplicial $G$-complex of dimension at most $n$. If there was a continuous $G$-equivariant from $||\mathcal{K}||$ to $||\mathcal{L}||$, then there would be a continuous $G$-equivariant map from $||\Delta(P(\mathcal{K}))||$ to $||\Delta (P(\mathcal{L}))||$. This comes from the fact that for every finite simplicial $G$-complex $\mathcal{F}$, the $G$-space $||\mathcal{F}||$ is $G$-homeomorphic to the $G$-space $||\Delta(P(\mathcal{F}))||$. Now, 
$$\chi\left(\widetilde{C}_{P(\mathcal{L})}\right)\leq\dim (P(\mathcal{L}))+1=\dim (\mathcal{L})+1\leq n+1,$$
and $P(\mathcal{K})$ is $n$-connected, which contradicts the part $(a)$ of the previous theorem. In addition, in light of the following facts, one can easily check that the part $(b)$ implies $G$-Tucker's lemma. 
\begin{enumerate}
\item The space $||\Delta ({\left(G\cup\{0\}\right)
}^n\setminus\{(0,\ldots , 0)\})||$ has the homotopy type of a wedge of\linebreak $(n-1)$-dimensional spheres~\cite[section 6.2]{matousek2008using}. Therefore $||\Delta\left({\left(G\cup\{0\}\right)
}^n\setminus\{(0,\ldots , 0)\}\right)||$ is $(n-2)$-connected.
\item The strong compatibility graph of $(G\times [n-1], \preceq_1)$, $\widetilde{C}_{G\times [n-1]}$, is isomorphic with the complete $(n-1)$-partite graph $K_{\underbrace{|G|,\ldots, |G|}_{n-1}}$. Thus, 
$$\chi\left(\widetilde{C}_{G\times [n-1]}\right)= n-1\quad\text{as}\quad\chi\left(K_{\underbrace{|G|,\ldots, |G|}_{n-1}}\right)= n-1.$$
\item Distinct elements $(g, x)$ and $(h, y)$ are comparable in $(G\times [n-1], \preceq_1)$ if and only if $x\neq y$.
\end{enumerate}
In the following, we provide a series of examples to show that one can hope to infer much more information from the generalized Dold's theorem than the ordinary one. Actually, for this purpose, we need to construct free $G$-posets whose the chromatic number of their strong compatibility graphs are substantially smaller than their dimensions. Let us start with a simple example.    
\begin{example}
\normalfont
In this example, we equip $G\times [n]$ with another ordering, called $\preceq_2$, rather than $\preceq_1$.
Let $R= G\times [n]$ be the $G$-poset whose the action is given $h\cdot (g, i)\to (hg, i)$, and the order is defined by $(g, x) \prec_2 (h, y)$ if $x < y$ and $g=h$. It is easy to see that $\dim(R) = n-1$. Now, in the following we  will show that 
$$\chi (\widetilde{C}_{R}) =\min\{\left| G\right|, n\}.$$ Put $m = \min\{\left| G\right|, n\}$. Let $x_1 <\cdots < x_m$ and $g_1, \ldots, g_m$ be distinct elements of $[n]$ and $G$, respectively. By the definition of $\widetilde{C}_{R}$, it is easily seen that the vertices 
$$(g_1, x_1), (g_2, x_2), \ldots, (g_m, x_m)$$ form a clique of size $m$ in $\widetilde{C}_{R}$. Thus, $m\leq\chi\left(\widetilde{C}_{R}\right)$. For the reverse side, on the one hand by Lemma $7$ 
$$\chi\left(\widetilde{C}_{R}\right)\leq\dim(R)+1=n.$$ 
On the other hand, the following map gives a proper coloring of $\widetilde{C}_{R}$ with $|G|$-colors.
\begin{align*}
 c : \widetilde{C}_{R}\longrightarrow & \{1, \ldots, \left| G\right|\}\\
  & (h, y) \longmapsto h.
\end{align*}
Therefore, $\chi (\widetilde{C}_{R}) =\min\{\left| G\right|, n\}$. So, for this case, whenever $n >> \left| G\right|$, there is a huge gap between the amount of information which is released from the generalized Dold theorem than the original one.
\end{example}
For constructing some more interesting examples, we need to recall the definition of Hom-complexes which were introduced by Lov\'{a}sz.
\begin{definition}
Let $F$ and $H$ be graphs and $V(F)=[n]$. The Hom-poset ${Hom}_{p}(F, H)$ is a poset whose elements are given by all $n$-tuples $(A_1, \cdots, A_n)$ of non-empty subsets of $V(H)$ with the property that for any edge $\{i, j\}\in E(F)$, and every $x\in A_i$, $y\in A_j$ we have $\{x, y\}\in E(H)$. The partial order is defined by $A=(A_1,\cdots, A_n)\leq B=(B_1,\cdots, B_n)$ if and only if $A_i\subseteq B_i$ for all $i\in[n]$. Now, the order complex of $Hom_{p}(F, H)$ is called the Hom-complex, and is dented by $Hom(F, H)$. 
\end{definition}
Let us see another example here with the aid of Hom-complexes. 
\begin{example}
\normalfont
Let $C_{r}$ be a cycle of length $r$ with the vertex set $[r]$ and the edge set
$\{\{i, i+1\}: 1\leq i\leq r-1\}\cup\{\{1, r\}\}$, and $K_n$ be the complete graph with $n$ vertices. The group $\mathbb{Z}_2=\{e=\omega^0, \omega\}$ acts on $Hom_{p}(C_{r}, K_n)$ as follows, 
$$\omega\cdot (A_1, A_{2}, \ldots, A_{r-1},A_{r})=(A_{r},A_{r-1},\ldots , A_{2}, A_{1}).$$ 
\normalfont
Clearly, $Hom_{p}(C_{r}, K_n)$ with this action is a free $\mathbb{Z}_2$-poset. It is not hard to see that $\chi\left(\widetilde{C}_{Hom_{p}(C_{r}, K_n)}\right)\leq n$. Indeed, we can color any vertex $(A_1, \ldots, A_r)$ of $\widetilde{C}_{Hom_{p}(C_{r}, K_n)}$ with an arbitrary element of $A_1$. Note that if $A=(A_1,\ldots, A_r)$ is adjacent to $B=(B_1, \ldots , B_r)$ in $\widetilde{C}_{Hom_{p}(C_{r}, K_n)}$, then we must have $A\preceq\omega\cdot B$ or $B\preceq\omega\cdot A$. Without loss of generality, assume that $A\preceq\omega\cdot B$. In particular, this means that $A_1\subseteq B_r$. On the other hand we have $B_1\cap B_r=\emptyset$ as $\{1,r\}$ is an edge of $C_{r}$. Therefore, $A_1\cap B_1=\emptyset$. So, each of them receives different colors and hence this coloring is a proper coloring. Thus, $\chi\left(\widetilde{C}_{Hom_{p}(C_{r}, K_n)}\right)\leq n$. However, the dimension of $Hom_p(C_r, K_n)$ can be arbitrary far away from $n$. Indeed, we always have $\dim\left(Hom_p(C_r, K_n)\right)\geq\lceil\frac{r}{2}\rceil(n-2)$. To see this, we need to find a chain of the mentioned length in $Hom_{p}(C_r, K_n)$. If $r$ is even, we can start the chain with the element $(\{1\},\{2\}, \ldots, \{1\},\{2\})$, and in each steep add one element from the set $[r]\setminus\{1,2\}$ to one of them that already contains $1$ until we reach to the element $([r]\setminus\{2\},\{2\},[r]\setminus\{2\},\ldots, [r]\setminus\{2\},\{2\})$, i.e.,
\begin{align*}
& (\{1\},\{2\}, \ldots, \{1\},\{2\})\prec (\{1, 3\},\{2\}, \ldots, \{1\},\{2\})\prec\cdots\prec ([r]\setminus\{2\},\{2\}, \ldots, \{1\},\{2\})\prec\\
& ([r]\setminus\{2\},\{2\},\{1,3\} \ldots, \{1\},\{2\})\prec\cdots\prec ([r]\setminus\{2\},\{2\},[r]\setminus\{2\},\ldots, \{1\},\{2\})\prec\ldots\\
& \ldots\prec ([r]\setminus\{2\},\{2\},[r]\setminus\{2\},\ldots, [r]\setminus\{2\},\{2\}).
\end{align*}
For the case that $r$ is odd, we just need to start by $(\{1\},\{2\},\{1\},\{2\}, \ldots, \{1\},\{2\},\{3\})$ and the rest of the procedure is the same as described above.
\end{example}
As we saw in above, for some graphs $F$ the Hom-poset $Hom_p(F, H)$ can be equipped with a free action. As another example, let $F=K_{r}$ be the complete graph with the vertex set $[r]$. The cyclic group $\mathbb{Z}_r=\{e=\omega^0, \omega,\ldots, \omega^{r-1}\}$ can act on the poset $Hom_{p}(K_r, H)$ naturally by cyclic shift. In other words, for each $\omega^{i}\in\mathbb{Z}_r$ and $(A_1,\cdots, A_r)\in {Hom}_{p}(K_r, H)$, define $\omega^{i}\cdot (A_1,\cdots, A_r)=(A_{1+i(\text{mod}\, r)},\cdots , A_{r+i(\text{mod}\, r)})$. Obviously, this action turns the poset $Hom_{p}(K_r, H)$ to a free $\mathbb{Z}_r$-poset. From now on, we consider the Hom-posets  ${Hom}_{p}(K_r, H)$ as a $\mathbb{Z}_r$-poset with the mentioned $\mathbb{Z}_r$-actions above.

Now we are in a position to state our main example. In fact, we will show that for almost all graphs $H$ the chromatic number of the 
$\mathbb{Z}_r$-poset $Hom_p(K_r, H)$ is strictly smaller than its dimension. In particular, this shows that ``for almost all free $\mathbb{Z}_2$-space $X$" our estimation reveal more information than the ordinary Dold theorem; as every free $\mathbb{Z}_2$-space is $\mathbb{Z}_2$-homotopy equivalent to $Hom(K_2, H)$ for some graph $H$~\cite{csorba2007homotopy}. To state our result precisely, we need to review some definitions and facts from the theory of random graphs. Given $n$ and $p$, a random graph $G(n, p)$ is a graph with labeled vertex set $[n]$, where each edge appears independently with probability $p$. We say that a random graph $G(n, p)$ has a property $\mathcal{P}$ almost surely (or with high probability) if 
$$P(G(n, p)\,\text{has property}\,\mathcal{P})\to 1\quad\text{as}\quad n\to\infty.$$ 
Now, we are in a position to state our result precisely.
\begin{theorem}
For any real number $\epsilon > 0$ and any fixed natural number $r\geq 2$ the following inequality holds almost surely 
 $$\chi\left(\widetilde{C}_{Hom (K_r, G(n,\frac{1}{2}))}\right)\leq
        \left(r+\frac{2}{r-1}+ \epsilon\right)\log_2 n.$$
\end{theorem}
Let us begin with the following lemma, which plays a key role in the proof.
\begin{lemma}
Let $H$ be a graph and $r\geq2$ be a positive integer. If\hspace{0.2cm}$H$ contains no copy of $K_{\underbrace{t,\ldots, t}_{r}}$, then 
$$\chi\left(\widetilde{C}_{Hom(K_r, H)}\right)\leq\left(\binom{r}{2}+1\right)(t-1).$$
\end{lemma}
\begin{proof}
For simplicity, we first introduce some notation.
For a vertex $A=(A_1, \ldots, A_r)$ in $\widetilde{C}_{Hom_p(\mathcal{K}_r, H)}$, put
\begin{align*}
m_A & =\min\{|A_i| : 1\leq i\leq r\},\\
\Gamma_A & = \{i : |A_i|= m_A, 1\leq i\leq r\},\\
\lambda_A & = j\quad\Longleftrightarrow\quad \min\left(\bigcup_{i\in\Gamma_A}A_i\right)\in A_j.
\end{align*}
In the following lines, we define a proper coloring of $\widetilde{C}_{Hom_p(K_r, H)}$. Let $X=(X_1, \ldots, X_r)$ be a vertex of $\widetilde{C}_{Hom_p(K_r, H)}$. Define 
\[
c (X)=
\begin{cases}
(m_X, |\Gamma_X|, \lambda_X)\quad & \text{if}\quad |\Gamma_X|< r\\
m_X\quad & \text{if}\quad |\Gamma_X|=r
\end{cases}
\]
First note that this map uses at most $\left(\binom{r}{2}+1\right)(t-1)$ number of colors. Indeed, if $1\leq |\Gamma_X| < r$, then $1\leq m_X\leq (t-1)$, and $1\leq\lambda_X\leq |\Gamma_X|$, and for the case that $|\Gamma_X| =r$ we have at most $(t-1)$ possibility for $m_X$. Now, we need to show that this coloring is proper, i.e, no two adjacent vertices receives the same color. Suppose $A=(A_1, \ldots, A_r)$ and $B = (B_1, \ldots, B_r)$ are adjacent in $\widetilde{C}_{Hom_p(K_r, H)}$. Thus, there is an $1\leq i\leq r-1$ such that $A\subseteq\omega^i\cdot B$ or $B\subseteq\omega^i\cdot A$. Without loss of generality, we can assume that $A\subseteq\omega^i\cdot B$ $(1)$. Now, we consider two cases. The first one is the case that $m_A=m_B$, and $|\Gamma_A|=|\Gamma_B| < r$ $(2)$. To confirm our claim, we need to show that $\lambda_A\neq\lambda_B$. The properties $(1), (2)$ imply $\Gamma_A =\{b+i \Mod r: b\in \Gamma_B\}$, and $A_j=B_{j+i(\text{mod}\, r)}$ for all $j\in\Gamma_A$. Thus, $\lambda_A+i\equiv\lambda_B\Mod r$. This implies that $\lambda_A\neq\lambda_B$ as $1\leq i\leq r-1$. Now, consider the case that $|\Gamma_A|=r$. We claim that $|\Gamma_B|\neq r$. Suppose, contrary to our claim, that $|\Gamma_A|=|\Gamma_B|=r$. This property beside the fact that $A\subseteq\omega^i\cdot B$ imply that $A= \omega^i\cdot B$. Therefore, $A$ and $B$ lie in a same orbit in $Hom_p(K_r, H)$  and hence they are not connected in $\widetilde{C}_{Hom_p(K_r, H)}$. This contradiction finishes the proof.
\end{proof}
Now we are in a position to prove Theorem $13$. Actually, in light of Lemma $14$, it is enough to find the minimum $r$ such that no copy of $K_{\underbrace{t,\ldots, t}_{r}}$ appears in $G(n,\frac{1}{2})$ with high probability. This part may be proved in much the same way as Corollary 2.4 in~\cite{Kahle}. However, for the convenience of the reader we give a proof here.
\begin{proof1}
\normalfont
Let $t =\lceil{a\log_2 n\rceil}$ where $a=\frac{r}{\binom{r}{2}}(1 +\epsilon)$ and $\epsilon > 0$. We show that no copy of $K_{\underbrace{t,\ldots, t}_{r}}$ appears in $G(n,\frac{1}{2})$ with high probability. Suppose that $U_1,\ldots , U_r$ are pairwise disjoint subsets of $[n]$, each of size $t$. The probability that the complete $r$-partite graph with parts $U_1,\ldots , U_r$ appears in $G(n, \frac{1}{2})$ is ${\frac{1}{2}}^{\binom{r}{2}t^2}$. Moreover, the number of complete $r$-partite graphs where each part has size $t$ on the vertex set $[n]$ is $\frac{\binom{n}{t}\binom{n-t}{t}\cdots\binom{n-(r-1)t}{t}}{r!}$.
Thus, the probability that any $K_{t,\ldots, t}$ appears in $G(n,\frac{1}{2})$ is bounded above by 
$$\frac{\binom{n}{t}\binom{n-t}{t}\cdots\binom{n-(r-1)t}{t}}{r!}{\frac{1}{2}}^{\binom{r}{2}t^2}.$$
On the other hand
\begin{equation*} 
\begin{split}
 \left(\frac{\binom{n}{t}\binom{n-t}{t}\cdots\binom{n-(r-1)t}{t}}{r!}\right)\left({\frac{1}{2}}^{\binom{r}{2}t^2}\right) & \leq \\
 \left(n^{rt}\right)\left({\frac{1}{2}}^{\binom{r}{2}t^2}\right)   & \leq\\
 \left(n^{r\left(a\log_2 n+1\right)}\right)\left({\frac{1}{2}}^{\binom{r}{2}a^2{\left(\log_2 n\right)}^2}\right) & =\\
 \left(n^{r\left(a\log_2 n+1\right)}\right)\left({\left(\underbrace{{({2^{-1})}^{\log_2 n}}}_{n}\right)}^{\binom{r}{2}a^2{\log_2 n}}\right) & =\\ 
 n^{r+ ra\log_2 n\left(1-\frac{a}{r}\binom{r}{2}\right)} & = n^{r-ra\epsilon\log_2 n}\to 0\quad\text{as}\quad n\to\infty.
\end{split}
\end{equation*}
Now, using Lemma 14, one can say that the following inequality holds with high probability
\begin{align*}
\chi\left(\widetilde{C}_{Hom_p(K_r, G(n,\frac{1}{2}))}\right) & \leq\left(1+\binom{r}{2}\right)\frac{r}{\binom{r}{2}}(1 +b\epsilon)\log_2 n\quad\\
& = \left(r+\frac{2}{r-1}+ \epsilon\right)\log_2 n
\end{align*}
where $b=\frac{1}{\left(1+\binom{r}{2}\right)\frac{r}{\binom{r}{2}}}$. Now, the proof is completed.\qed
\end{proof1}

Now, for a comparison let us compute the asymptotic behavior of the dimension of $Hom_p(K_r, G(n,\frac{1}{2}))$. Clearly,
If a graph $H$ contains a copy of complete $r$-partite subgraph $K_{l_1,\ldots,l_r} $, then the dimension of $Hom_p(K_r, H)$ is at least $l_1+\ldots+l_r-r$. 
We claim that for any $0< \beta < 1$ a copy of 
$K_{\alpha,\underbrace{1,\ldots,1}_{r-1}}$ will appear in $G(n, \frac{1}{2})$ with high probability, where $\alpha=\lceil\frac{\beta n}{2^{r-1}}\rceil$. In particular, this implies with high probability
$\dim\left(Hom(K_r, G(n,\frac{1}{2}))\right)\geq\frac{n}{2^{r}}$,
which is much larger than the asymptotic value of $\chi\left(\widetilde{C}_{Hom (K_r, G(n,\frac{1}{2}))}\right)$! Now, to prove our claim, consider the following events:
\begin{itemize}
    \item Let $A_n$ be the event that $G(n,\frac{1}{2})$ contains a copy of $K_{\alpha,\underbrace{1,\ldots,1}_{r-1}}$.
    \item Let $B_n$ be the event that $G(n,\frac{1}{2})$ contains a clique of size $r-1$.
    \item Let $C_n$ be the event that any subset of vertices of $G(n,\frac{1}{2})$  of size $r-1$ has at least $\alpha$ common neighbors in $G(n,\frac{1}{2})$.
\end{itemize}
We must show that $P(A_n)$ tends to one as $n$ goes to infinity. First of all we have,
$P(A_n)\geq P(B_n\cap C_n)$.
And since $A_n$ and $B_n$ are monotone increasing, by FKG inequality~\cite{Frieze}[Section 21.3] we have $P(B_n\cap C_n)\geq P(B_n)P(C_n)$. So, to confirm our claim, it is enough to show that the both of $P(B_n)$ and $P(C_n)$ tends to $1$ as $n$ goes to infinity. The first one is an easy task as $n^{-\frac{2}{r-1}}$ is a threshold\cite{Frieze}[Theorem 5.3] for appearing $K_{r-1}$ in $G(n,p)$. To prove the other one, letting $\overline{C_n}$ the event that $C_n$ does not occur. Now, we show that $P(\overline{C_n})$ goes to zero as $n$ goes to infinity.  For any fixed set $A\subseteq [n]$ of size $r-1$, let $X_{A}$ be the number of common neighbors of $A$ in $G(n,\frac{1}{2})$. The probability that some other vertex of $G(n,p)$ is adjacent to all element of $A$ is $\frac{1}{2^{r-1}}$. Also, note that $X_A$ has binomial distribution with parameters $(n-(r-1),\frac{1}{2^{r-1}})$. In particular, $\mu =E(X_A)=\frac{n-(r-1)}{2^{r-1}}$. Chose $0 <\delta < 1$ in a way that $\alpha\leq (1-\delta)\mu$. Now, using Chernoff bound~\cite{Frieze}[Corollary 21.7], we have
$$P(X_A\leq\alpha)\leq P(X_A\leq(1-\delta)\mu)\leq e^{-\frac{\delta^2\mu}{2}}$$
Thus,
\begin{equation*} 
\begin{split}
 P(\overline{C_n})  & \leq \sum_{A\subseteq [n], |A|=r-1}P(X_A\leq (1-\delta)E(X_A))\\
& \leq\binom{n}{r-1}e^{-\frac{\delta^2\mu}{2}}\\
& \leq n^{r-1}e^{-\delta^2\frac{n-(r-1)}{2^r}}\to 0\quad\text{as}\quad n\to\infty 
\end{split}
\end{equation*}
Therefore, for any $0< \beta < 1$, $K_{\frac{\beta n}{2^{r-1}},\underbrace{1,\ldots,1}_{r-1}}$ appears in $G(n, \frac{1}{2})$ with high probability. 
\vspace{0.2cm}

We will finish the paper by mentioning two other applications of strong compatibility graphs.
\subsection{Application I: Limitations of topological bounds}
In fact, the motivation behind the definition of Hom-complexes goes back to a well-known conjecture of Kneser about the chromatic number of special family of graphs, called Kneser graphs. Surprisingly, this conjecture was solved via introducing a topological lower bound for the chromatic number of graphs by Lov\'{a}sz~\cite{Lovasz}. Actually, the Lov\'{a}sz bound in the setting of $\mathbb{Z}_2$-index is as follows:
\begin{theorem}[Lov\'{a}sz~\cite{Lovasz}]
For every graph $H$ we have
$$\chi(H)\geq Ind_{\mathbb{Z}_2}||Hom(K_2, H)||+2.$$
\end{theorem}
However, this bound is tight for Kneser graphs. But, we cannot expect that it always works well. For example, it is known that if a graph $H$ does not contains a copy of complete bipartite graph $K_{l,m}$, then its Lov\'{a}sz bound is at most $l+m-1$; see the $K_{l,m}$ theorem in~\cite{Csorba}. While, the chromatic number of such graphs can be arbitrary large~\cite{erdos}. As another evidence, M. Kahle~\cite{Kahle} showed that for any $\epsilon > 0$, the following inequality holds almost surly
$$Ind_{\mathbb{Z}_2} ||Hom(K_2, G(n,\frac{1}{2}))||\leq (4+\epsilon)\log_2 n.$$ 
While, for comparison, the chromatic number of $G(n,\frac{1}{2})$ is tightly concentrated around $\frac{n}{2log_2 n}$~(\cite{bollobas})! Later, the Lov\'{a}sz bound was generalized to all complete graphs by Babson and Kozlov.
\begin{theorem}[\cite{Kozlov}]
For every graph $H$ we have
$$\chi(H)\geq Ind_{\mathbb{Z}_2}||Hom(K_r, H)||+r.$$
\end{theorem}
Now, as a direct corollary of Theorem $8$ and Lemma $14$ we have the following generalization of the $K_{l,m}$ theorem (for the case $l=m$).
\begin{corollary}
Let $H$ be a graph and $r\geq2$ be a positive integer. If\hspace{0.2cm}$H$ contains no copy of $K_{\underbrace{t,\ldots, t}_{r}}$, then 
$$Ind_{\mathbb{Z}_r}||Hom (K_r, H)||\leq\left(\binom{r}{2}+1\right)(t-1)-1.$$
\end{corollary}
As another corollary of Theorem 8 and Theorem 13, we can generalize the Kahle result to all complete graphs as follows. 
\begin{corollary}
For any real number $\epsilon > 0$ and any natural number $r\geq 2$ the following inequality 
\[
    Ind_{\mathbb{Z}_r}||Hom (K_r, G(n,\frac{1}{2}))||\leq\left(r+\frac{2}{r-1}+ \epsilon\right)\log_2 n
\]        
holds almost surely. 
\end{corollary}

\subsection{Application II: Constructing triangle-free high chromatic number graphs from spheres}
As a final application, we shall explain how one can construct triangle-free graphs with high chromatic number from the $n$-sphere $\mathbb{S}^{n}$! let $\mathbb{Z}_2=\{+1, -1\}$ be the cyclic group of order $2$.
\begin{lemma}
Let $P$ be a free $\mathbb{Z}_2$-poset. The compatibility graph $C_P$ of $P$ is triangle-free if and only if there are no $x, y\in P$ such that $x\preceq y$ and $x\preceq -y$.
\end{lemma}
\begin{proof}
Suppose there exist $x, y\in P$ such that $x\preceq y$ and $x\preceq -y$. First, we show that $x$, $y$, and $-y$ are distinct elements of $P$. First of all $y\neq -y$ as $P$ is free. Also $x\neq -y$, since otherwise $-y\preceq y$ as $x\preceq y$. By multiplication of both sides of the inequality $-y\preceq y$ by $-1$, we get $y\preceq -y$ as well. Thus, by transitivity of $\preceq$, $y=-y$ which contradicts the fact that $P$ is a free $\mathbb{Z}_2$-poset. Similarly, $x\neq y$. Therefore, they are pairwise distinct. Moreover, $x\sim y$, $x\sim -y$, and $y\sim -y$ as $x\preceq -y$, $x\preceq y$, and $y\preceq y$, respectively. Therefore $x$, $y$, and $-y$ form a triangle in $C_P$.

For the other side, assume that $x, y,$ and $z$ form a triangle in $C_P$. Since $x$ is connected to $y$, $x$ must be comparable by $-y$, i.e, $x\preceq -y$ or $y\preceq -x$. We have the similar conditions for the other pairs; $x, z$, and $y, z$. Without loss of generality, we can just consider two cases.
\\
\textbf{Case I:}
\[   
     \begin{cases}
       x\preceq -y\\
       y\preceq -z\\
       z\preceq -x\
     \end{cases}.
\]
Therefore, by transitivity of $\preceq$, we have $x\preceq -x$. Thus, $x=-x$, in the same way as in the first part. This contradicts the fact that $P$ is a free $\mathbb{Z}_2$-poset.
\\
\textbf{Case II:}
\[   
     \begin{cases}
       x\preceq -y\\
       x\preceq -z\\
       y\preceq -z\
     \end{cases}.
\]
The first and third inequalities imply that $x\preceq z$. On the other hand, by the second inequality, we have also $x\preceq -z$ which contradicts our assumptions.
\end{proof}
It is easy to check that the face poset $P(\mathcal{K})$ of any free $\mathbb{Z}_2$-simplicial complex $\mathcal{K}$ satisfies in the conditions of Lemma $19$. Therefore, $C_{P(\mathcal{K})}$ is a triangle-free graph. Consequently, the strong compatibility graph $\widetilde{C}_{P(\mathcal{K})}$ is triangle free as $\widetilde{C}_{P(\mathcal{K})}$ is a sub-graph of $C_{P(\mathcal{K})}$. Moreover, Lemma 7 and Theorem 8 imply $\chi\left(\widetilde{C}_{P(\mathcal{K})}\right)=n+1$, if $\dim(K)=Ind_{\mathbb{Z}_2}(||K||)=n$. In particular, we have the following corollary as $Ind_{\mathbb{Z}_2}\mathbb{S}^n= n$.
\begin{corollary}[Triangle-free graphs with high chromatic number]
If $\mathcal{K}$ is a free $\mathbb{Z}_2$-simplicial complex which triangulates the $n$-sphere $\mathbb{S}^n$,i.e, $||\mathcal{K}||\cong\mathbb{S}^n$, then $\widetilde{C}_{P(\mathcal{K})}$ is a triangle-free graph with $\chi (\widetilde{C}_{P(\mathcal{K})})=n+1$.
\end{corollary}

\section*{Acknowledgements}
I wish to express my most sincere gratitude to Professors Hossein Hajiabolhassan, and Roman Karasev, for his kind guidance, helpful comments and suggestions. A special thanks goes to Professors Maksim Zhukovskii and Omid Etesami for fruitful discussions about Random graphs. Finally, I would also like to thank Professor Andrey Raigorodskii for his generous support and encouragement. 

\section*{References}

\end{document}